\documentclass[10pt,letterpapper]{amsart}

\usepackage{graphicx}
\usepackage[colorlinks = true,
  linkcolor = blue,
  urlcolor  = blue,
  citecolor = blue,
  anchorcolor = blue]{hyperref}
\usepackage[hmargin=1in, vmargin=1in]{geometry}
\renewcommand{\theequation}{\thesection.\arabic{equation}}
\usepackage{subfig}

\theoremstyle{plain}
\newtheorem*{main}{Main Theorem}
\newtheorem{thm}{Theorem}
\newtheorem{lem}{Lemma}
\newtheorem{prop}[lem]{Proposition}
\newtheorem{cor}[lem]{Corollary}
\theoremstyle{definition}

\newtheorem{rmk}{Remark}

\newcommand{\din}{\mathrm{in}}
\newcommand{\dout}{\mathrm{out}}

\renewcommand{\[}{\begin{equation}\notag\begin{aligned}}
\renewcommand{\]}{\end{aligned}\end{equation}}
\newcommand{\beq}[1]{\begin{equation}\label{#1}\begin{aligned}}
\newcommand{\beqeps}[1]%
  {\addtocounter{equation}{1}\begin{equation}\label{#1}\tag{$\theequation\epsilon$}\begin{aligned}}

\newcommand{\p}{\begin{pmatrix}}
\newcommand{\pp}{\end{pmatrix}}

\author{Ting-Hao Hsu}
\address{Department of Mathematics\\ The Ohio State University\\ Columbus, OH 43210}
\email{hsu.296@osu.edu}
\title[Bifurcation Delay Using Geometric Singular Perturbation Theory]%
{On Bifurcation Delay: An Alternative Approach Using\\ Geometric Singular Perturbation Theory}


\begin{document}

\maketitle
\begin{center}
\today
\end{center}

\begin{abstract}
To explain the phenomenon of bifurcation delay,
which occurs in planar systems
of the form $\dot{x}=\epsilon f(x,z,\epsilon)$, $\dot{z}=g(x,z,\epsilon)z$,
where $f(x,0,0)>0$ and $g(x,0,0)$ changes sign at least once on the $x$-axis,
we use the Exchange Lemma 
in Geometric Singular Perturbation Theory
to track the limiting behavior of the solutions.
Using the trick of extending dimension
to overcome the degeneracy at the turning point,
we show that the limiting attracting and
repulsion points are given by the well-known entry-exit function,
and the maximum of $z$ on the trajectory
is of order $\exp(-1/\epsilon)$.
Also we prove smoothness the return map
up to arbitrary finite order in $\epsilon$.
\end{abstract}

\section{Introduction} \label{sec_intro}
Consider the planar system \beqeps{sf_xz}
  &\dot{x}= \epsilon f(x,z,\epsilon)\\
  &\dot{z}= g(x,z,\epsilon)z
\] with $x\in \mathbb R$, $z\in \mathbb R$,
$f$ and $g$ are $C^1$ functions satisfying \beq{cond_fg}
  f(x,0,0)>0;\quad
  g(x,0,0)<0 \;\;\text{for}\; x<0
  \quad\text{and}\quad
  g(x,0,0)>0 \;\;\text{for}\; x>0.
\]

\begin{figure}[t]
{\centering
\subfloat[\label{fig_xz0}]{
\boxed
{\includegraphics[trim = 2.4cm 8cm 2cm 7cm, clip, width=.48\textwidth]{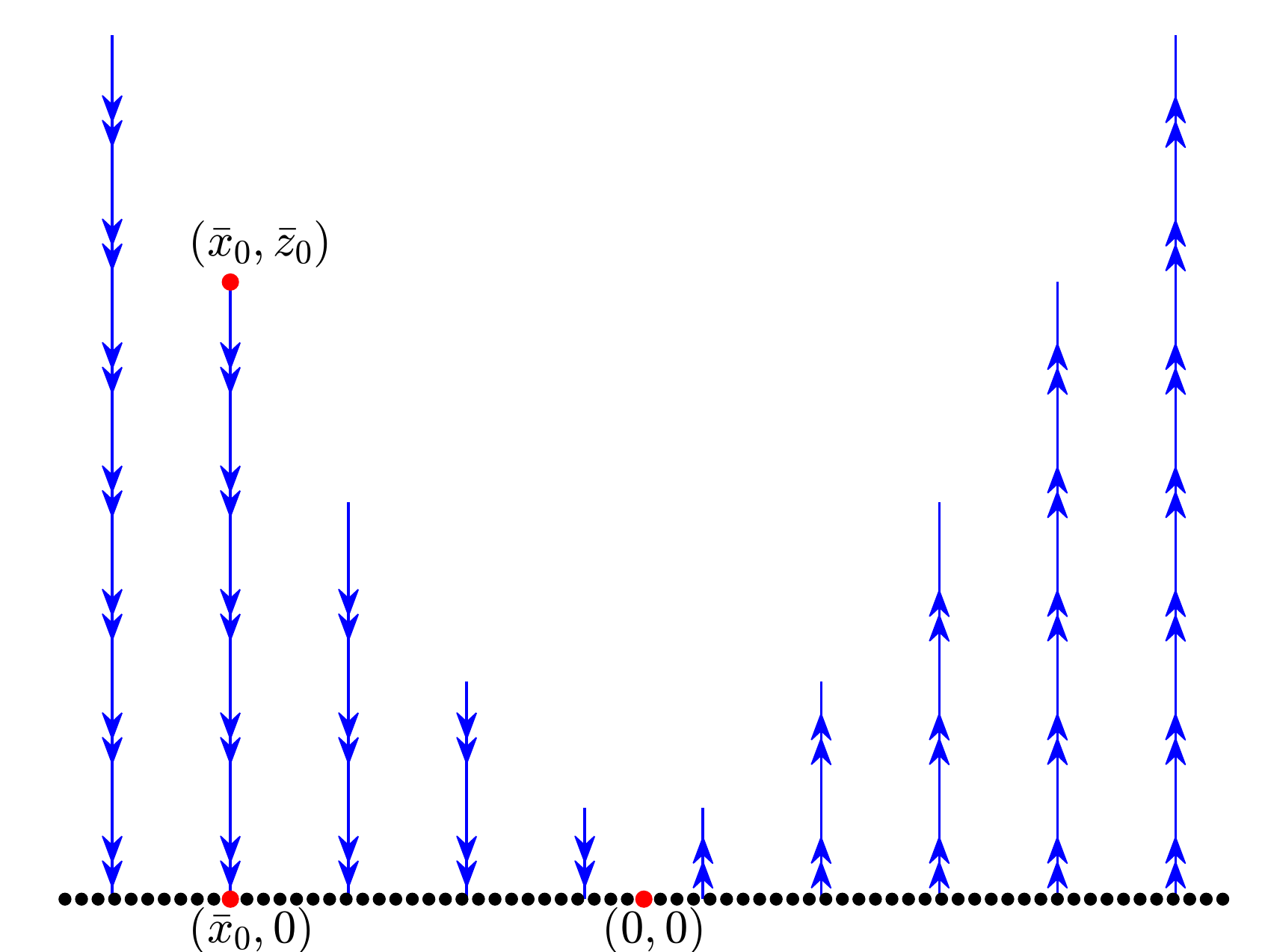}}
}
\subfloat[\label{fig_xz1}]{
\boxed
{\includegraphics[trim = 2.4cm 8cm 2cm 7cm, clip, width=.48\textwidth]{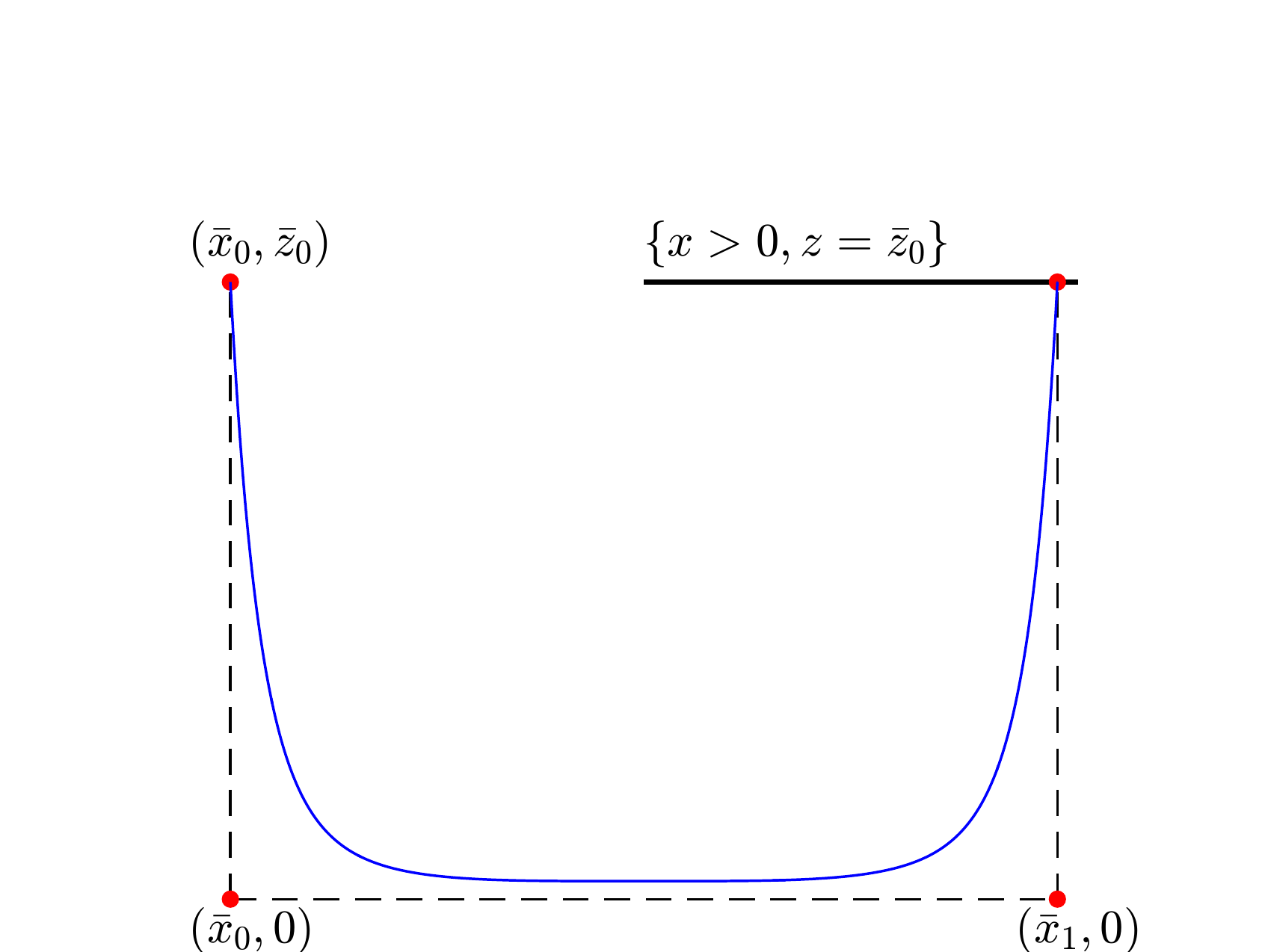}}
} \\
}
\caption{
(a) When $\epsilon=0$, the $x$-axis is a line of equilibria for \eqref{sf_xz}.
The trajectory starting at $(x_0,z_0)$ goes straight down to the $x$-axis.
Separated by the turning point $(x,z)=(0,0)$,
the $x$-axis changes from attracting to repelling.
(b) When $\epsilon>0$ and small,
the trajectory of \eqref{sf_xz} starting at $(x_0,z_0)$
first tends to the $x$-axis,
and then it turns at $(x_0,0)$.
The trajectory turns again when it approaches the point $(x_1,0)$
satisfying $\int_{x_0}^{x_1}\frac{g(x,0,0)}{f(x,0,0)}\;dx=0$.
}
\label{fig_xz}
\end{figure}

Note that \eqref{sf_xz} is a \emph{slow-fast} system
\cite{Jones:1994,Kuehn:2016}
with fast variable $z$
and slow variable $x$.
Fix any $x_0<0$ and choose $z_0>0$ small enough
so that $g(x_0,z,0)<0$ for all $z\in [0,z_0].$
When $\epsilon=0$, it is clear that the trajectory starting at $(x_0,z_0)$
goes straight to $(x_0,0)$.
The $x$-axis is attracting when $x<0$ and repelling when $x>0$
since $g(x,0,0)$ changes sign at $x=0$.
For $\epsilon>0$,
besides being attracted by the $x$-axis,
the trajectory also moves right at speed of order $\epsilon$.
After the trajectory passes $x=0$,
the $x$-axis becomes repelling,
so the trajectory tends to move away from the $x$-axis.
See Fig \ref{fig_xz0}.
However, it is well known that,
for small $\epsilon>0$,
the trajectory does not immediately leave the vicinity of the $x$-axis after crossing the origin.
Instead, the trajectory stays at the $x$-axis until it approaches
the point $(x_1,0)$ which satisfies \beq{defn_x1}
  \int_{x_0}^{x_1}\frac{g(x,0,0)}{f(x,0,0)}\;dx=0.
\]
See Fig \ref{fig_xz1}.
This phenomenon has been called
``bifurcation delay" \cite{Benoit:1991},
``Pontryagin delay" \cite{Mishchenko:1994},
or ``delay of instability" \cite{Liu:2000},
and the function $x_0\mapsto x_1$ implicitly defined by \eqref{defn_x1}
is called the entry-exit \cite{Benoit:1981}
or way in-way out \cite{Diener:1984} function.

Bifurcation delay
has been studied by various methods in the literature,
including asymptotic expansion \cite{Haberman:1979,Mishchenko:1994},
comparison to solutions constructed by separation of variables \cite{Schecter:1985},
gradient estimates using the variational equation \cite{De-Maesschalck:2008a},
and the blow-up method of geometric singular perturbation theory \cite{De-Maesschalck:2016}.

Our results stated below are included in the literature,
but the proof in this note
provides a new approach
using geometric singular perturbation theory.

\begin{main}
Consider \eqref{sf_xz}, where $f$ and $g$ are $C^{r+1}$, $r\ge 1$,
and satisfy \eqref{cond_fg}.
Choose $x_0< 0$ such that
there exists $x_1>0$ satisfying \eqref{defn_x1}.
If $z_0>0$ is small enough, then the following holds.
Let $\bar{\gamma}_\epsilon$ be the trajectory of \eqref{sf_xz}
that starts at the point $(x_0,z_0)$
and ends at the cross section $\{x>0, z=z_0\}$.
Then \beq{closeness_gamma}
  \bar{\gamma}_\epsilon
  \to \bar{\gamma}_1\cup \bar{\gamma}_0\cup \bar{\gamma}_2
  \equiv
  \Big( \{x_0\}\times [0,z_0] \Big)
  \cup \Big( [x_0,x_1]\times \{0\} \Big)
  \cup \Big( \{x_1\}\times [0,z_0] \Big)
  \quad\text{as}\quad
  \epsilon\to 0
\] in the sense of point-sets, and \beq{min_z}
  \min_{(x,z)\in \bar{\gamma}_\epsilon} z
  = \exp\left(\frac{-\zeta_0+o(1)}{\epsilon}\right),
  \quad
  \text{where}\quad
  \zeta_0=\int_{x_0}^0 \frac{|g(x,0,0)|}{f(x,0,0)}\; dx.
\]
Moreover,
for any compact interval $K\subset (-\infty,0)$
such that
$x_1$ is well-defined by \eqref{defn_x1} for each $x_0\in K$,
there exist $\epsilon_0>0$
such that if we set $x_{1,\epsilon}$ by $x_{1,0}=x_1$ and \[
  (x_{1,\epsilon},z_0)
  = \bar{\gamma}_\epsilon\cap \{x>0,z=z_0\},
  \quad \epsilon\in (0,\epsilon_0],
\] then $x_{1,\epsilon}$
is a $C^r$ function of $(x_0,\epsilon)$ on $K\times [0,\epsilon_0]$.
\end{main}

In Section \ref{sec_sing_config}
we analyze the structure of \eqref{sf_xz}
as a slow-fast system.
In Section \ref{sec_EL}
we state a simple case of the Exchange Lemma
which can be applied in our context.
In Section \ref{sec_complete}
we complete the proof of the Main Theorem.

\section{Singular Configuration} \label{sec_sing_config}
Fix $z_0>0$ small enough so that $g$ is negative on $\bar{\gamma}_1$
and is positive on $\bar{\gamma}_2$,
where $\bar{\gamma}_i$ are defined in \eqref{closeness_gamma}.

From the equation of $\dot{x}$ in \eqref{sf_xz},
to prove \eqref{closeness_gamma}
we expect the travel time of $\gamma_\epsilon$ to be of order $1/\epsilon$,
so we set $\tau=\epsilon t$, where $t$ is the time variable in \eqref{sf_xz},
and expect the change in $\tau$ along the $\gamma_\epsilon$ to be of order $1$.
Fix any $\Delta>0$ with $\Delta<\frac12\min\{|x_0|,x_1\}$.
Set \beq{defn_Ieps}
  \mathcal{I}_\epsilon = \left\{
    \p x \\ z \\ \tau \pp
    = \p x_0 \\ z_0 \\ \sigma \pp:\;
    |\sigma|<\Delta
  \right\},\quad
  \mathcal{J}_\epsilon = \left\{
    \p x \\ z \\ \tau \pp
    = \p x_1+\sigma \\ z_0 \\ \tau_1 \pp:\;
    |\sigma|<\Delta
  \right\},
\] where $\tau_1>0$ is defined later in \eqref{defn_tau1}.
Our strategy is to show that the manifolds evolved from
$\mathcal{I}_\epsilon$ and $\mathcal{J}_\epsilon$ along the flow \eqref{sf_xz}
have nonempty intersection.

\begin{figure}[t]
{\centering
{\includegraphics[trim = 2cm 7.8cm 2cm 6.8cm, clip, width=.5\textwidth]
{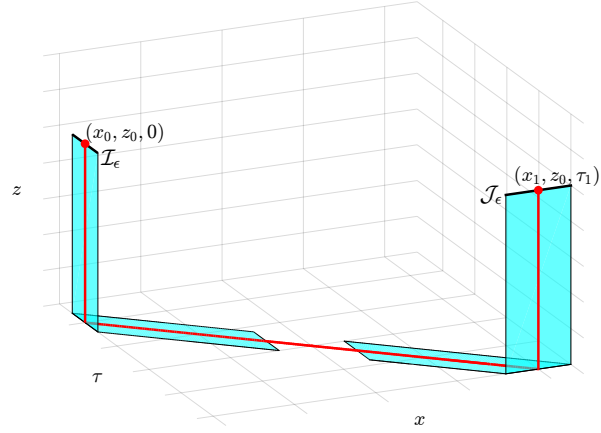}}\\
}
\caption{
The manifolds evolved from $\mathcal{I}_\epsilon$ and $\mathcal{J}_\epsilon$
are two surfaces nearly parallel to $\{z=0\}$ in $(x,z,\tau)$-space.
}
\label{fig_parallel}
\end{figure}

To track the slow time variable $\tau$,
we append the equation $\dot{\tau}=\epsilon$ to the system \eqref{sf_xz}.
That is, we write \eqref{sf_xz} as \beqeps{sf_tau}
  &\dot{x}= \epsilon f(x,z,\epsilon)\\
  &\dot{z}= g(x,z,\epsilon)z \\
  &\dot{\tau}= \epsilon.
\] Intuitively, when $\epsilon>0$ is small,
$\mathcal{I}_\epsilon$ and $\mathcal{J}_\epsilon$ first go straight to the plane $\{z=0\}$,
and then go along the flow in $(x,\tau)$-plane.
Although the two surfaces evolved from $\mathcal{I}_\epsilon$ and $\mathcal{J}_\epsilon$
may go along close to a common trajectory on the plane,
they are both nearly parallel to $\{z=0\}$.
Thus it is not clear whether they intersect;
see Fig \ref{fig_parallel}.
To overcome this difficulty, we introduce the new variable \beq{defn_zeta}
  \zeta= \epsilon\log(1/z).
\] 
Using $\dot{\zeta}=-\epsilon \dot{z}/z$,
the system \eqref{sf_tau} can be expressed as \beqeps{sf_xzeta}
  &\dot{x}= \epsilon f(x,z,\epsilon) \\
  &\dot{z}= g(x,z,\epsilon)z \\
  &\dot{\zeta}= -\epsilon g(x,z,\epsilon) \\
  &\dot{\tau}= \epsilon.
\]
Note that \eqref{sf_xzeta} a slow-fast system,
which has two distinguished limiting systems.
The \emph{limiting fast system} (or \emph{layer problem}),
obtained by setting $\epsilon=0$ in \eqref{sf_xzeta}, 
is \beq{fast_xzeta}
  &\dot{z}= g(x,z,0)z \\
  &\dot{x}= 0,\;
  \dot{\zeta}= 0,\;
  \dot{\tau}= 0,
\] and the \emph{limiting slow system} (or \emph{reduced problem}) is \beq{slow_xzeta}
  &z=0 \\
  &x'= f(x,0,0) \\
  &\zeta' = -g(x,0,0) \\
  &\tau'= 1,
\] where $\prime$ is $\frac{d}{d\tau}$.
The spirit of Geometric Singular Perturbation Theory
is to first study the limiting systems,
which have no parameter $\epsilon$ and have lower dimension,
and then make conclusion about the full system for $\epsilon>0$.

In $(x,z,\zeta,\tau)$-space,
$\mathcal{I}_\epsilon$ and $\mathcal{J}_\epsilon$ are parametrized as \beq{defn_Ieps}
  \mathcal{I}_\epsilon= \left\{
    \p x \\ z \\ \zeta \\ \tau \pp
    = \p x_0 \\ z_0 \\ 0 \\ 0 \pp
    +\sigma \p 0 \\ 0 \\ 0 \\ 1 \pp
    +\epsilon \p 0 \\ 0 \\ \log(1/z_0) \\ 0 \pp
    :\; 
    |\sigma|< \Delta
  \right\}
\]
and \beq{defn_Jeps}
  \mathcal{J}_\epsilon= \left\{
    \p x \\ z \\ \zeta \\ \tau \pp
    = \p x_1 \\ z_0 \\ 0 \\ \tau_1 \pp
    +\sigma \p 1 \\ 0 \\ 0 \\ 0 \pp
    +\epsilon \p 0 \\ 0 \\ \log(1/z_0) \\ 0 \pp 
    :\;
    |\sigma|< \Delta
  \right\}.
\]

From now on we identify the $(x,\zeta,\tau)$-space
with the set $\{z=0\}$ in $(x,z,\zeta,\tau)$-space (and temporarily ignore the relation \eqref{defn_zeta} between $z$ and $\zeta$).
Let $\Lambda_L$ and $\Lambda_R$ be the
$\omega$- and $\alpha$-limit sets of
$\mathcal I_0$ and $\mathcal J_0$ along the flow of \eqref{fast_xzeta}. 
Then \beq{defn_LambdaLR}
  \Lambda_L= \left\{
    \p x \\ \zeta \\ \tau \pp
    = \p x_0 \\ 0 \\ \sigma \pp:\;
    |\sigma|< \Delta
  \right\}, \quad
  \Lambda_R= \left\{
    \p x \\ \zeta \\ \tau \pp
    = \p x_1+\sigma \\ 0 \\ \tau_1 \pp:\;
    |\sigma|< \Delta
  \right\}.
\]
The following proposition describes
the manifolds evolved from $\Lambda_L$ and $\Lambda_R$
along the slow flow \eqref{slow_xzeta}.

\begin{prop}\label{prop_MLR}
Assume $x_0$ and $x_1$ satisfy \eqref{defn_x1}.
Set \beq{defn_tau1}
  \tau_1
  = \int_{x_0}^{x_1}\frac{1}{f(x,0,0)}\;dx.
\] Let $M_L$ and $M_R$
be the manifolds evolved from $\Lambda_L$ and $\Lambda_R$
along the slow flow \eqref{slow_xzeta}.
Then $M_L$ and $M_R$
intersect transversally
along a curve $\gamma_0$
in $(x,\zeta,\tau)$-space.
Moreover, the projection of $\gamma_0$ in $(x,z)$-space is
$\bar{\gamma}_0$ defined in \eqref{closeness_gamma},
and the minimum of $\zeta$ on $\gamma_0$ equals $\zeta_0$ defined in \eqref{min_z}.
\end{prop}

See Fig \ref{fig_MLR} for an illustration of Proposition \ref{prop_MLR}.

\begin{proof}
By integrating \eqref{slow_xzeta},
a portion of $M_L$ can be parametrized as \beq{para_ML}
  M_L = \left\{
    \p x\\ \zeta \\ \tau \pp
    =\p x\\ \zeta_-(x;\hat\tau_0) \\ \tau_-(x;\hat\tau_0) \pp:\;
    \begin{array}{l}
    x\in [x_0,x_1]\\[.5em]
    \hat\tau_0\in [-\Delta,\Delta]
    \end{array}
  \right\}
\] where \beq{defn_zetaL}
  \zeta_-(x;\hat\tau_0)= \int_{x_0}^x \frac{-g(r,0,0)}{f(r,0,0)}\;dr,\quad
  \tau_-(x;\hat\tau_0)= \hat\tau_0+ \int_{x_0}^x \frac{1}{f(r,0,0)}\;dr,
\] and a portion of $M_R$ can be parametrized as \beq{para_MR}
  M_R= \left\{
    \p x\\ \zeta \\ \tau \pp
    =\p x\\ \zeta_+(x;\hat{x}_1) \\ \tau_+(x;\hat{x}_1) \pp:\;
    \begin{array}{l}
    x_1\in [\hat{x}_1-\Delta,\hat{x}_1+\Delta] \\[.5em]
    x\in [x_0,x_1]
    \end{array}
  \right\}
\] where \beq{defn_zetaR}
  \zeta_+(x;\hat{x}_1)= \int_{x}^{\hat{x}_1} \frac{g(r,0,0)}{f(r,0,0)}\;dr,\quad
  \tau_+(x;\hat{x}_1)= \tau_1- \int_x^{\hat{x}_1} \frac{1}{f(r,0,0)}\;dr.
\]
From the assumption \eqref{defn_x1} and the definition \eqref{defn_tau1},
we have \beq{eq_zeta_pm}
  \zeta_-(0;0)= \int_{x_0}^0\frac{|g(r,0,0)|}{f(r,0,0)}\;dr
  =\zeta_+(0;x_1),\quad
  \tau_-(0;0)= \int_{x_0}^0\frac{1}{f(r,0,0)}\;dr
  = \tau_+(0;x_1).
\] From uniqueness of solution of \eqref{slow_xzeta},
it follows that
$M_L$ and $M_R$ intersect along the curve \beq{defn_gamma0}
  \gamma_0
  = \left\{
    \p x\\ \zeta \\ \tau \pp
    =\p x\\ \zeta_-(x;0) \\ \tau_-(x;0) \pp
    =\p x\\ \zeta_+(x;x_1) \\ \tau_+(x;x_1) \pp:\;
    x\in [x_0,x_1]
  \right\}.
\] Clearly the projection of $\gamma_0$ in $(x,z)$-space is $\bar{\gamma}_0$
and the minimum of $\zeta$ on $\gamma_0$ is ${\zeta}_0$ defined in \eqref{min_z}.
It remains to show that the intersection is transversal.

Fix any $\hat{x}\in [x_0,x_1]$.
Let $\hat{q}=\gamma_0\cap \{x=\hat{x}\}$.
Taking derivatives in $x$ and $\tau_0$
in \eqref{para_ML}, we obtain \[
  T_{\hat{q}}M_L= \mathrm{Span}\;\left\{
    \p f(\hat{x},0,0) \\ -g(\hat{x},0,0) \\1 \pp,
    \p 0 \\ 0 \\ 1 \pp
  \right\}.
\]
Similarly, taking derivatives in $x$ and $x_1$ in \eqref{para_MR} we obtain \[
  T_{\hat{q}}M_R= \mathrm{Span}\;\left\{
    \p f(\hat{x},0,0) \\ -g(\hat{x},0,0) \\1 \pp,
    \p 0 \\ g(x_1,0,0) \\ -1 \pp
  \right\}.
\] Since $f(\hat{x},0,0)> 0$ and $g(x_1,0,0)> 0$,
the union of $T_{\bar{q}}M_L$ and $T_{\bar{q}}M_R$
spans the $(x,\zeta,\tau)$-space.
This means the intersection is transversal.
\end{proof}

For the limiting fast system \eqref{fast_xzeta},
the forward trajectory of $(x_0,z_0,0,0)$
and the backward trajectory of $(x_1,z_0,0,0)$ 
are \beq{defn_gamma12}
  \gamma_1
  = \left\{
    \p x \\ z \\ \zeta \\ \tau \pp
    = \p x_0 \\ z \\ 0 \\ 0 \pp :\;
      z\in [0,z_0]
  \right\}
  \quad\text{and}\quad
  \gamma_{2}
  = \left\{
    \p x \\ z \\ \zeta \\ \tau \pp
    = \p x_1 \\ z \\ 0 \\ 0 \pp :\;
      z\in [0,z_0]
  \right\},
\] respectively.
Combined with $\gamma_0$ given in Proposition \ref{prop_MLR},
now we have the configuration \beq{config}
  (x_0,z_0,0,0)
  \overset{\gamma_1}{\longrightarrow}
  (x_0,0,0,0)
  \overset{\gamma_0}{\longrightarrow}
  (x_1,0,0,0)
  \overset{\gamma_2}{\longrightarrow}
  (x_1,z_0,0,0),
\] where $\gamma_1$ and $\gamma_2$ are
trajectories of the limiting fast system \eqref{fast_xzeta},
and $\gamma_0$ is a trajectory of 
the limiting slow system \eqref{slow_xzeta}.

Note that for each $\gamma_i$ ($i=0,1,2$),
the projection in $(x,z)$-space is $\bar{\gamma}_i$ defined in \eqref{closeness_gamma},
and the projection in $(x,z,\tau)$-space is shown in Fig \ref{fig_parallel}.
In $(x,\zeta,\tau)$-space,
$\gamma_1$ and $\gamma_2$ collapse to points,
and $\gamma_0$ is as shown in Fig \ref{fig_MLR}.

To track the manifolds
evolved from $\mathcal I_\epsilon$ and $\mathcal J_\epsilon$,
we make the following simple observation.

\begin{figure}[t]
{\centering
{\includegraphics[trim = 2cm 7.8cm 2cm 6.8cm, clip, width=.5\textwidth]
{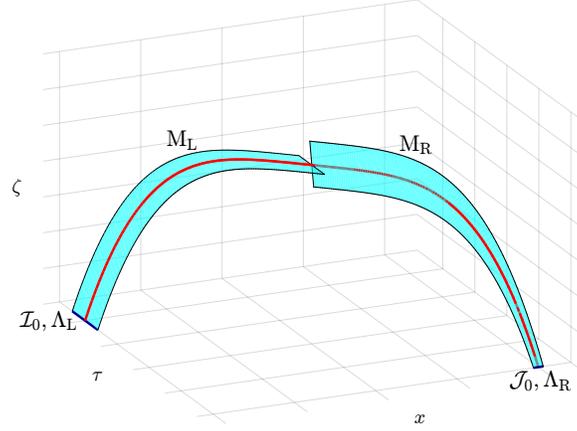}}\\
}
\caption{
The projections of $\mathcal{I}_0$ and $\mathcal{J}_0$
are $\Lambda_L$ and $\Lambda_R$.
The manifolds $M_{L,R}$ evolved from $\Lambda_{L,R}$
intersect transversally along a curve
$\gamma_0$ in $(x,\zeta,\tau)$-space.
}
\label{fig_MLR}
\end{figure}

\begin{prop}\label{prop_zzeta}
Consider the system \eqref{sf_xzeta}
in $(x,z,\zeta,\tau)$-space.
Suppose $\epsilon>0$.
Let $\mathcal{I}_\epsilon^*$ and $\mathcal{J}_\epsilon^*$
be the manifolds evolved from $\mathcal I_\epsilon$ and $\mathcal J_\epsilon$.
Then the following statements are equivalent. \begin{itemize}
  \item[$\mathrm{(i)}$] $\mathcal{I}_\epsilon^*$ and $\mathcal{J}_\epsilon^*$
  have nonempty intersection
  in $(x,z,\zeta,\tau)$-space.
  \item[$\mathrm{(ii)}$]
  The projections of $\mathcal{I}_\epsilon^*$ and $\mathcal{J}_\epsilon^*$
  in $(x,z,\tau)$-space
  have nonempty intersection.
  \item[$\mathrm{(iii)}$]
  The projections of $\mathcal{I}_\epsilon^*$ and $\mathcal{J}_\epsilon^*$
  in $(x,\zeta,\tau)$-space
  have nonempty intersection.
\end{itemize}
\end{prop}

\begin{proof}
On $\mathcal{I}_\epsilon^*$ and $\mathcal{J}_\epsilon^*$
the relation $z=\epsilon\log(1/\zeta)$ holds,
so these statements are equivalent.
\end{proof}

From this proposition,
to prove \eqref{closeness_gamma}
we only need to check that
the projections of $\mathcal{I}_\epsilon^*$ and $\mathcal{J}_\epsilon^*$
in $(x,\zeta,\tau)$-space have nonempty intersection.
This will be confirmed in Section \ref{sec_complete}
using the Exchange Lemma stated in Section \ref{sec_EL}.

\section{A Simple Case of the Exchange Lemma} \label{sec_EL}
To track manifolds for the system \eqref{sf_xzeta},
we will apply the following simple case of the Exchange Lemma.

\begin{thm}\label{thm_EL}
Consider a system for
$(b,c)\in \mathbb R\times \mathbb R^l$, $l\ge 1$,
of the form
\beqeps{sf_bc}
  &\dot b= \rho(b,c,\epsilon)b\\
  &\dot c= \epsilon h(c,\epsilon)
\]
defined on a set of the form $
  {}[0,\Delta]\times \overline{U},
$
where $\Delta>0$ and $U$ is a bounded open set in $\mathbb R^l$.
Assume the coefficients are $C^r$ functions, $r\ge 1$, satisfying \beq{cond_h}
  \rho(b,c,\epsilon)<-\nu
  \quad\text{and}\quad
  |h(c,\epsilon)|>\nu
\] for some $\nu>0$.
Let $\Lambda$
be a $(\sigma-1)$-dimensional compact manifold in $\mathbb R^l$, $1\le \sigma\le l$.
Suppose for each $c\in \Lambda$
we have $h(c)\notin T_{c}\Lambda$.
That is, the slow flow \beq{slow_bc}
  c'= h(c,0)
\] is not tangent to $\Lambda$.
Fix any $\tau_+>\tau_->0$ satisfying \[
  \Lambda\circ [0,\tau_+]\subset U,
\] where $\circ$ is the solution operator for \eqref{slow_bc},
and set $M= \Lambda\circ [\tau_-,\tau_+]$.
Then there exists a positive number $\Delta_1<\Delta$ such that 
for any $b_0\in (0,\Delta_1)$
and $C^r$ function $\varphi(c,\epsilon)$,
if we set \beq{def_Ieps_EL}
  \mathcal I_\epsilon^\din
  = \{
  (b,c)= 
  \big(
  b_0,\hat{c}+\epsilon\varphi(\hat{c},\epsilon)
  \big): \hat{c}\in \Lambda
  \big\},
\]
then there is on neighborhood $V$ of $M$ such that \beq{I_close}
  \text{
    ${\mathcal I}_\epsilon^*\cap V$ 
    is $C^r$ $O(\epsilon)$-close to $M$,
  }
\] where ${\mathcal I}_\epsilon^*$ is the manifold evolved from $\mathcal I_\epsilon^\din$
along \eqref{sf_bc}.
\end{thm}

\begin{proof}
The assertion \eqref{I_close}
with $r=1$
follows from the $(k+\sigma)$-Exchange Lemma \cite[Theorem 6.7]{Jones:2009}
with $(k,m)=(0,1)$.
The $C^r$ smoothness follows from
the General Exchange Lemma \cite{Schecter:2008b}.
\end{proof}

See Fig \ref{fig_EL_bc} for an illustration of the theorem.
The assertion \eqref{I_close}
means that there is a $C^r$ function $\tilde{b}(c,\epsilon)$
defined on $M\times [0,\epsilon_0]$,
where $\epsilon_0>0$,
satisfying \[
  (\tilde{b}(c,\epsilon),c)\in {\mathcal I}_\epsilon^*
  \quad\forall\; c\in M,\; \epsilon\in (0,\epsilon_0]
\] and $\|\tilde{b}(\cdot,\epsilon)\|_{C^r(M)}= O(\epsilon)$.

\begin{rmk}
The theorem
can also be interpreted as a special case
of the General Exchange Lemma in Schecter \cite{Schecter:2008b},
the \emph{Strong $\lambda$-Lemma} in Deng \cite{Deng:1989},
or the \emph{$C^r$-Inclination Theorem} in Brunovsky \cite{Brunovsky:1999}.
\end{rmk}

\begin{rmk}\label{rmk_smooth}
Let $\{\Lambda_{\mu}\}_{\mu\in A}$,
where $A$ is a compact interval,
be a $C^r$ family of $(\sigma-1)$-dimensional manifolds. 
If we replace $\Lambda$ by $\Lambda_\mu$
and denote the corresponding $\mathcal I_\epsilon$ by $\mathcal I_{\epsilon,\mu}$.
Then,
as a result of the General Exchange Lemma \cite{Schecter:2008b},
${\mathcal I}_{\epsilon,\mu}^*\cap V$
is uniformly $C^r$ for $(\mu,\epsilon)\in A\times (0,\epsilon_0]$.
\end{rmk}

\begin{figure}[t]
{\centering
{\includegraphics[trim = 2cm 7.8cm 2cm 6.8cm, clip, width=.5\textwidth]
{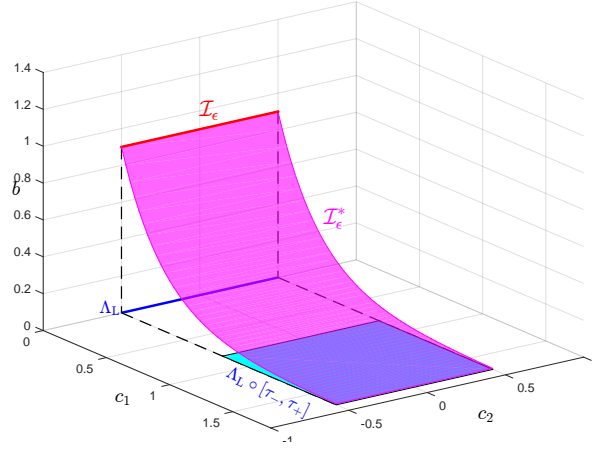}}\\
}
\caption{
Consider the system $(\dot{b},\dot{c}_1,\dot{c}_2)=(-b,\epsilon,0)$.
Suppose $\mathcal I_\epsilon^\din=\{b_0\}\times \Lambda$.
Then the projection of $\mathcal I_\epsilon^\din$ along the stable fibers is $\Lambda$.
The Exchange Lemma assures that $\mathcal{I}_\epsilon^*$,
the manifold evolved from $\mathcal I_\epsilon^\din$,
is $C^1$ $O(\epsilon)$-close to $\Lambda\circ [\tau_-,\tau_+]$,
where $\circ$ is the solution operator for the slow flow $(c_1',c_2')=(1,0)$.
}
\label{fig_EL_bc}
\end{figure}

\begin{cor}\label{cor_EL}
Assume all the assumptions in Theorem \ref{thm_EL}
except we replace \eqref{sf_bc} by \beqeps{sf_bc_perturbed}
  &\dot b= \rho(b,c,\epsilon)b\\
  &\dot c= \epsilon h(c)+ \epsilon b E(b,c,\epsilon),
\]
where the coefficients are $C^{r+1}$ functions, $r\ge 1$.
Then conclusion in the theorem still holds.
Moreover, the mapping \[
  \Pi_\epsilon:
  \mathcal I_\epsilon^*\cap V
  \to \mathcal I_\epsilon\times [\tau_-,\tau_+],\quad
  q_\epsilon\mapsto (q^\din_\epsilon,\tau_\epsilon),
\] defined by \[
  q_\epsilon= q_\epsilon^\din\cdot (\tau_\epsilon/\epsilon),
\]
where $\cdot$ is the solution operator for \eqref{sf_bc_perturbed},
is uniformly $C^r$ in $\epsilon$.
\end{cor}

\begin{proof}
By Fenichel's Theorem \cite{Fenichel:1979},
there exists $\Delta_1>0$ 
and a $C^{r+1}$ function $\pi^+_\epsilon$ of the form \[
  \pi^+_\epsilon(b,c)
  = \big(0,c+\epsilon b \theta(b,c,\epsilon)\big)
\] defined on a neighborhood of \[
  {}[0,\Delta_1]\times (\Lambda\circ [0,\tau_+])
\]
such that \beq{pi_fenichel}
  \pi^+_\epsilon\Big(
  (b,c)\cdot t
  \Big)
  = \Big(
  \pi^+_\epsilon
  \big(
  b,c)
  \big)
  \Big)\cdot t,\quad t\ge 0,
\] where $\cdot$ is
the solution operator of \eqref{sf_bc_perturbed}.
This implies that the new coordinates \beq{def_tilde_bc}
  \big(\tilde b,\tilde c\big)
  = \big(b,c-\epsilon b \theta(b,c,\epsilon)\big)
\] brings \eqref{sf_bc_perturbed} into a system of the form \beqeps{sf_bc1}
  &\dot {\tilde{b}}= \tilde\rho(\tilde b,\tilde c,\epsilon)\tilde b\\
  &\dot {\tilde c}= \epsilon \tilde{h}(\tilde c,\epsilon),
\]
where $\tilde\rho$ and $\tilde h$ are $C^r$ functions,
and $\tilde h$ satisfies \beq{tilde_h0}
  \tilde h(\tilde c,\epsilon)\big|_{b=0,\epsilon=0}
  = h(c,0).
\] 
Decrease $\Delta_1$ if necessary so that
the conclusion in Theorem \ref{thm_EL} holds.
Choose any $b_0\in (0,\Delta_1]$ and $C^{r+1}$ function $\varphi(c,\epsilon)$,
and set $\mathcal I_\epsilon^\din$ by \eqref{def_Ieps_EL}.
Then we can write $\mathcal{I}_\epsilon^\din$ as \[
  \mathcal I_\epsilon^\din
  = \{(b,c): (\tilde{b},\tilde{c})
  =\big(b_0,\hat{c}+\epsilon\tilde\varphi(\hat{c},\epsilon)\big): \hat{c}\in \Lambda\}
\] for some $C^{r+1}$ function $\tilde{\varphi}$.
By Theorem \ref{thm_EL} and \eqref{tilde_h0}
there exists a neighborhood $V$ of \[
  \tilde{M}\equiv \{(b,c): (\tilde b,\tilde c)\in \Lambda \circ [\tau_-,\tau_+]\},
\] where $\circ$ is the solution operator of \eqref{slow_bc},
such that \[
  \text{
  $\mathcal{I}_\epsilon^*\cap V$
  is $C^r$ $O(\epsilon)$-close to
  $\tilde{M}$
  }
\] From \eqref{def_tilde_bc} we know
$\tilde{M}$ is $C^{r+1}$ $O(\epsilon)$-close to $M$,
so we obtain \eqref{I_close}.

Given $q_\epsilon=(b_\epsilon,c_\epsilon)\in \mathcal I_\epsilon^*\cap V$,
write $(q_\epsilon^\din,\tau_\epsilon)
=((b_0,c_\epsilon^\din),\tau_\epsilon)=\Pi_\epsilon(q)$.
Let \[
  \pi^+_\epsilon(b_\epsilon,c_\epsilon)
  = (0,c_\epsilon^1)
  \quad\text{and}\quad
  \Lambda_\epsilon
  = \pi^+_\epsilon(\mathcal I_\epsilon^\din)
\]
Since $\pi^+_\epsilon$ is uniformly $C^r$
and the flow $c'= h(c,\epsilon)$ is non-tangential to $\Lambda_\epsilon$,
we can uniquely define a uniformly $C^r$ function
$c_\epsilon^1\mapsto (c_\epsilon^0,\tau_\epsilon)$
by \[
  c_\epsilon^1= c_\epsilon^0\bullet \tau_\epsilon,
  \quad c_\epsilon^0\in \Lambda_\epsilon,
\] where $\bullet$ is the solution operator for $c'= h(c,\epsilon)$.
From \eqref{pi_fenichel} we have \[
  c_\epsilon^\din
  = \left(
  (\pi^+_\epsilon)^{-1}(0,c_\epsilon^0)
  \right) \cap \{b=b_0\}.
\]
Note that the restriction of $\pi^+_\epsilon$ on $\{b=b_0,c\in V\}$
is a local diffeomorphism
into $\{b=0,c\in V\}$.
Hence the function $c_\epsilon^0\mapsto c_\epsilon^\din$
is well-defined and is uniformly $C^r$.

Consider the function $q_\epsilon\mapsto (c_\epsilon^\din,\tau_\epsilon)$
as the composition of the following sequence \[
  q_\epsilon
  \mapsto c_\epsilon^1
  \mapsto (c_\epsilon^0,\tau_\epsilon)
  \mapsto (c_\epsilon^\din,\tau_\epsilon)
  \mapsto (c_\epsilon^\din,\tau_\epsilon)
\] We have seen that
each mapping in the sequence is uniformly $C^r$ in $\epsilon$.
Hence the function
$q_\epsilon\mapsto q_\epsilon^\din=((b_0,c_\epsilon^\din),\tau_\epsilon)$ is uniformly $C^r$.
\end{proof}

\section{Completing the Proof of the Main Theorem} \label{sec_complete}
Let $\mathcal{I}_\epsilon^*$ and $\mathcal{J}_\epsilon^*$
be the manifolds evolved from
$\mathcal{I}_\epsilon$ and $\mathcal{J}_\epsilon$, respectively,
defined in \eqref{defn_Ieps}-\eqref{defn_Jeps},
along the flow \eqref{sf_xzeta}.

\begin{prop}\label{prop_Imid}
Fix any positive number $\delta<\min\{\frac{|x_0|-\Delta}4,\frac{x_1-\Delta}4\}$.
Let \[
  \gamma_0^\delta
  = \gamma_0\cap \{x\in [x_0+\delta,x_1-\delta]\}
  \quad\text{and}\quad
  M_{L,R}^\delta
  = M_{L,R}\cap \{x\in [x_0+\delta,x_1-\delta]\}.
\] Then there is
a neighborhood $V$ of $\gamma_0^\delta$
such that \beq{I_close_Mdelta}
  \text{
  $\mathcal I_\epsilon^*\cap V$
  is $C^r$ $O(\epsilon)$-close to
  $M_L^\delta\cap V$
  }
\] and \beq{J_close_Mdelta}
  \text{
  $\mathcal J_\epsilon^*\cap V$
  is $C^r$ $O(\epsilon)$-close to
  $M_R^\delta\cap V$.
  }
\] Moreover,
if we consider $\mathcal I_\epsilon= \mathcal I_{\epsilon,x_0}$
as a function of $x_0$
in the definition \eqref{defn_Ieps},
then there exists $\epsilon_0>0$ and $\delta_1>0$
such that \[
  \text{
  $\mathcal I_{\epsilon,x}^*\cap V$ is uniformly $C^r$ for
  $(x,\epsilon)\in [x_0-\delta_1,x_0+\delta_1]\times (0,\epsilon_0]$.
  }
\]
\end{prop}

\begin{proof}
Choose $\Delta_1>0$ so that
the conclusions of Corollary \ref{cor_EL} holds
for \eqref{sf_xzeta} restricted on \[
  z\in [0,\Delta_1],\quad
  \mathrm{dist}((x,\zeta,\tau),\Lambda_L)\le \delta,
\] (by the choice of $\delta$, this region lies in $\{x<0\}$)
with \[
  \Lambda= \Lambda_L,\quad
  \tau_-=\delta,\quad
  \tau_+= 2\delta.
\] Set \beq{defn_Iin}
  \mathcal I_\epsilon^\din
  = \mathcal I_\epsilon^*\cap \{x<0,z=\Delta_1\}.
\]
It is clear that \[
  \mathcal I_\epsilon^\din
  = \left\{
    \p x \\ z \\ \zeta \\ \tau \pp
    = \p x_0 \\ \Delta_1 \\ 0 \\ 0 \pp
    +\sigma \p 0 \\ 0 \\ 0 \\ 1 \pp
    +\epsilon \p \varphi_1 \\ 0 \\ \log(1/\Delta_1) \\ \varphi_2 \pp
    : |\sigma|<\Delta
  \right\}
\] for some $C^{r+1}$ functions $\varphi_i=\varphi_i(x,\tau,\epsilon)$.
By Corollary \ref{cor_EL},
there exists a neighborhood $\tilde V_L$ of \[
  M_{L}\cap \{x\in [x_0+\tfrac{\delta}2,x_0+2\delta]\}
\] such that \beq{I_close_Mdelta1}
  \text{
  $\mathcal I_\epsilon^*\cap V_L$
  is $C^1$ $O(\epsilon)$-close to
  $M_L\cap V_L$.
  }
\] 
In particular, setting \[
  \mathcal I_\epsilon^\delta
  \equiv \mathcal I_\epsilon^*\cap \{x=x_0+\delta\}
  \quad\text{and}\quad
  \bar{M}_L^\delta
  \equiv M_L\cap \{x=x_0+\delta\}
\] we have \beq{Iout_close}
  \text{
  $\mathcal I_\epsilon^\delta$
  is $C^1$ $O(\epsilon)$-close to
  $\bar{M}_L^\delta$.
  }
\] 
Next we will track $\mathcal I_\epsilon^*$
along $\gamma_0^\delta$
as the manifold evolved from $\mathcal I_\epsilon^\delta$.

Using the relation $z=\exp(-\zeta/\epsilon)$,
we write \eqref{sf_xzeta} as \beqeps{deq_xzeta}
  &\dot{x}= \epsilon f(x,e^{-\zeta/\epsilon},\epsilon) \\
  &\dot{\zeta}= -\epsilon g(x,e^{-\zeta/\epsilon},\epsilon) \\
  &\dot{\tau}= \epsilon.
\] By a rescaling of time, the system is equivalent to \beqeps{deq_xzeta_slow}
  &{x}'= f(x,e^{-\zeta/\epsilon},\epsilon) \\
  &{\zeta}'= - g(x,e^{-\zeta/\epsilon},\epsilon) \\
  &{\tau}'= 1.
\]
Note that the system \eqref{deq_xzeta_slow}
restricted in a neighborhood of $M_L^\delta$ 
is a regular perturbation of \eqref{slow_xzeta}
since $\inf_{M_L^\delta}\zeta>0$.
Also note that $M_L^\delta$
is evolved from $\hat M_L^\delta$ along \eqref{slow_xzeta}
and $\mathcal I_\epsilon^*$
is evolved from $\mathcal I_\epsilon^\delta$
along \eqref{deq_xzeta},
so it follows from \eqref{Iout_close} that
there is neighborhood $V_L$ of $M_L^\delta$
such that the following is true: \beq{I_close_MLdelta_p}
  \text{
  The projection of $\mathcal I_\epsilon^*\cap V_L$
  in $(x,\zeta,\tau)$-space
  is $C^r$ $O(\epsilon)$-close to $M_L^\delta$.
  }
\]
Using again the relation $z=\exp(-\zeta/\epsilon)$,
we then obtain \eqref{I_close_Mdelta} with $V$ replaced by $V_L$.
The $C^r$ smoothness of $\mathcal I_\epsilon^*\cap V_L$ in $(x_0,\epsilon)$,
as indicated in Remark \ref{rmk_smooth},
is a result of the General Exchange Lemma \cite{Schecter:2008b}.
Similarly,
\eqref{I_close_Mdelta} holds with $V$ replaced by
some neighborhood $V_R$ of $M_R^\delta$.
Now take $V=V_L\cap V_R$.
Then $V$ is a neighborhood of $\gamma_0^\delta$
and \eqref{I_close_Mdelta}-\eqref{J_close_Mdelta} hold.
\end{proof}

In Proposition \ref{prop_MLR}
we have seen that $M_L$ and $M_R$
intersect transversally along $\gamma_0$ in $(x,\zeta,\tau)$-space,
since transversal intersections persist under $C^1$ perturbation
(see e.g. \cite[Theorem 6.35]{Lee:2013}),
it follows from \eqref{I_close_Mdelta}-\eqref{J_close_Mdelta}
in Proposition \ref{prop_Imid} that
the projections of $\mathcal I_\epsilon^*\cap V$ and $\mathcal J_\epsilon^*\cap V$
in $(x,\zeta,\tau)$-space
intersect transversally along a curve $\gamma_\epsilon^\delta$
which is $C^1$ $O(\epsilon)$-close to $\gamma_0^\delta$
and is uniformly $C^r$ for $(x_0,\epsilon)$.
Using the relation $z=\exp(-\zeta/\epsilon)$,
as indicated in Proposition \ref{prop_zzeta},
we then recover an intersection curve $\gamma_\epsilon$
of ${\mathcal{I}}_\epsilon^*$ and ${\mathcal{J}}_\epsilon^*$
in $(x,z,\zeta,\tau)$-space.

From the uniqueness of solution of boundary value problems for \eqref{sf_xz},
the projection of $\gamma_\epsilon$ in $(x,z)$-space
is the trajectory $\bar{\gamma}_\epsilon$ defined in the statement of the Main Theorem.
By construction $\gamma_\epsilon$
lies in a $O(\epsilon)$-neighborhood of the configuration \eqref{config},
so \eqref{closeness_gamma} holds.

To prove the smoothness of the return map $(x_0,\epsilon)\mapsto x_{1,\epsilon}$,
it suffices to show that \[
  q_\epsilon\equiv\gamma_\epsilon\cap \{x>0,z=z_0\}
\]
is uniformly $C^r$ for $(x_0,\epsilon)\in [x_0-\Delta,x_0+\Delta]\times (0,\epsilon_0]$.
Let \[
  q_\epsilon^\delta\equiv \gamma_\epsilon\cap \{x=x_1-\delta\}
  \quad\text{and}\quad
  q_\epsilon^\dout\equiv \gamma_\epsilon\cap \{x>0,z=\Delta_1\}
\]
where $\Delta_1$ is small enough so that
the conclusion of Corollary \ref{cor_EL} holds near $\Lambda_R$.
Now consider the mapping $(x_0,\epsilon)\mapsto q_{1,\epsilon}$
as the composition of the following sequence: \[
  (x_0,\epsilon)
  \overset{\psi_1}{\longmapsto} p_\epsilon^\delta
  \overset{\psi_2}{\longmapsto} q_\epsilon^\dout
  \overset{\psi_3}{\longmapsto} q_{1,\epsilon}
\]
From Proposition \ref{prop_Imid}
we know $\psi_1$ is uniformly $C^r$.
From Corollary \ref{cor_EL}
we know $\psi_2$ is uniformly $C^r$.
On the other hand, it is clear that $\psi_3$ is uniformly $C^r$
since it is the Poincar\'{e} map
of a regularly perturbed flow along $\gamma_i$.
Hence we conclude that
the mapping $(x_0,\epsilon)\mapsto q_\epsilon$
is uniformly $C^r$ for $(x,\epsilon)\in [x_0-\Delta,x_0+\Delta]\times (0,\epsilon_0]$.
This proves the Main Theorem.

\section*{Acknowledgments}
The author would like to thank
Professors S.~Schecter 
and the anonymous referee's comments,
which helped to extend the results in this notes,
especially on the regularity of the return map.
The author would also like to thank
Professor X.-B.~Lin
for his insightful comments on this problem.


\end{document}